\documentclass[10pt]{amsart}

\usepackage{amssymb}
\usepackage{amsthm}
\usepackage{amsmath}

\newtheorem{theorem}{Theorem}[section]
\newtheorem{cor}[theorem]{Corollary}
\newtheorem{lemma}[theorem]{Lemma}

\newtheorem{defn}[theorem]{Definition}

\newcommand{\FF}[1]{\mathbb{F}_{#1}}

\begin{document}

\title{A new recursive construction for large sets of Kirkman triple systems}
\author{Chen Wang}
\address[Wang]{Johannes Kepler University Linz}
\email{chen.wang.1@outlook.com}

\begin{abstract}
We give a recursive construction of an LKTS$(7q+2)$ from an LKTS$(q+2)$
and an auxiliary permutation structure on an abelian group of order $q$, which
we call a cubic orthomorphism. We prove that cubic orthomorphisms exist whenever
$q$ is a product of prime powers congruent to $1$ modulo $6$. The construction
yields, in particular, LKTSs of orders $93,177,219,261$, and $303$.
\end{abstract}

\maketitle


\section{Introduction}
A \emph{Steiner triple system} of order $v$, denoted by STS$(v)$, is a pair $(X,\mathcal{B})$, where $X$ is a $v$-set and $\mathcal{B}$ is a collection of triples of $X$, called \emph{blocks}, such that every pair of points of $X$ is contained in a unique block of $\mathcal{B}$. Let $(X,\mathcal{B})$ be an STS$(v)$. If there exists a partition
\[
\mathcal{A}=\{P_1,P_2,\ldots,P_{(v-1)/2}\}
\]
of $\mathcal{B}$ such that every $P_i$ is a parallel class, that is, a partition of $X$, then the STS$(v)$ is called \emph{resolvable}, and $\mathcal{A}$ is called a \emph{resolution}. A resolvable STS$(v)$ is usually called a \emph{Kirkman triple system} of order $v$, briefly a KTS$(v)$. It is well known that a KTS$(v)$ exists if and only if $v\equiv3\pmod 6$.

Two STS$(v)$s on the same point set are said to be \emph{disjoint} if they have no blocks in common. By a simple counting argument, there can be at most $v-2$ mutually disjoint STS$(v)$s on the same point set. A set of $v-2$ mutually disjoint STS$(v)$s necessarily partitions all triples of the point set and is called a \emph{large set} of Steiner triple systems, denoted by LSTS$(v)$. A large set of Kirkman triple systems of order $v$, denoted by LKTS$(v)$, is an LSTS$(v)$ in which every constituent STS$(v)$ is resolvable.

The existence problem for LKTSs was posed by Sylvester in connection with Kirkman's schoolgirl problem. The existence problem for LSTSs was completely solved by Lu~\cite{MR692824,MR692825,MR692826,MR757612} and Teirlinck~\cite{MR1111565}, but the corresponding problem for LKTSs remains far from settled. Direct constructions are known only for a relatively sparse collection of orders. Table~\ref{tab:direct} is based on the survey of Chang and Zhou~\cite{MR3007142}, with the later additions from~\cite{MR3613568,MR3696119}.

\begin{table}[ht]
\centering
\caption{Known direct constructions of LKTSs.}
\label{tab:direct}
\begin{tabular}{p{3.7cm}rc}
\hline
Order & Discoverer & Year \\
\hline
$9$ & Kirkman~\cite{LKTS9} & 1850 \\
$15$ & Denniston~\cite{MR0369086} & 1974 \\
$33$ & Schreiber & 1974 \\
$51,75,105,129$ & Denniston~\cite{MR539718} & 1979 \\
$201,369$ & Chang and Ge~\cite{MR1675197} & 1999 \\
$273$ & Ge~\cite{MR2054972} & 2004 \\
$21,39$ & Zhou and Chang~\cite{MR2593325} & 2010 \\
$69,141,165,213,285$, & Zheng, Chang and Zhou~\cite{MR3613568} & 2016 \\
$309,333$ & & \\
$111,123,159,183,195$, & Zheng, Chang and Zhou~\cite{MR3696119} & 2017 \\
$231,279,291,339,381$ & & \\
\hline
\end{tabular}
\end{table}

The first recursive construction for LKTSs was a tripling construction of Denniston~\cite{MR535159}. It was later generalized to a product construction by Lei~\cite{MR1931492}, and the hypotheses of these constructions were weakened by Zhang and Zhu~\cite{MR1893029,MR1948399}. A different type of recursive construction was given by Wang and Shi~\cite{MR3612438}, who constructed an LKTS$(q^n+2)$ from an LKTS$(q+2)$ for suitable prime powers $q$. Further recursive constructions based on Steiner quadruple systems with resolvable derived designs and related structures were obtained by Xu and Ji~\cite{XuJi2021}, Liu and Lei~\cite{LiuLei2023}, and Tan and Zhou~\cite{TanZhou2025}.

The purpose of the present paper is to give another non-product construction. We introduce an auxiliary structure, called a cubic orthomorphism, on an abelian group of order $q$. Our main result, Theorem~\ref{thm:main}, states that the existence of an LKTS$(q+2)$ and a cubic orthomorphism of order $q$ implies the existence of an LKTS$(7q+2)$. We also show that cubic orthomorphisms exist whenever $q$ is a product of prime powers congruent to $1$ modulo $6$.

Combining the construction with the known LKTSs of orders $15,27,33,39$, and $45$ gives LKTSs of orders
$93,177,219,261$ and $303$.

The paper is organized as follows. In Section~2 we define cubic orthomorphisms and establish the properties needed in the construction. In Section~3 we construct the parallel classes, prove that they form KTSs of order $7q+2$, and then show that the resulting KTSs partition all triples.

\section{Auxiliary structures}
Let $q$ be a positive integer congruent to $1$ modulo $6$. The construction requires an abelian group of order $q$ equipped with the following additional structure.

\begin{defn}\label{defnCO}
A \emph{cubic orthomorphism} of order $q$, briefly a CO$(q)$, is a triple $(G,f,G^+)$, where $G$ is an abelian group of order $q$, $f$ is a permutation of $G$, and $G^+$ is a subset of $G$ such that
\[
\{0\}\mathbin{\dot\cup}G^+\mathbin{\dot\cup}(-G^+)=G,
\]
satisfying the following properties:
\begin{enumerate}
\item The map $a\mapsto f(a)-a$ is a bijection of $G$; in other words, $f$ is an orthomorphism.
\item We have $f(-a)=-f(a)$ for every $a\in G$. In particular, $f(0)=0$.
\item The permutation $f$ has cycle type $1^1 3^{(q-1)/3}$; equivalently, $f$ has order $3$ and has no fixed point other than $0$.
\item We have $f(G^+)=G^+$.
\item For every $a\in G$,
\[
a+f(a)+f^2(a)=0.
\]
\item For every $b\in G$, the map
\[
a\longmapsto f^2(a)+f(b-a)
\]
is a bijection of $G$.
\item For every $b\in G$, the two sets
\[
\{a+f(b)-f(a+b):a\in G^+\}
\]
and
\[
\{f^2(a)+f(b)-f(a+b):a\in G^+\}
\]
form a partition of $G\setminus\{0\}$.
\end{enumerate}
Here and throughout, $f^2(a)$ denotes $f(f(a))$.
\end{defn}

The definition is modeled on multiplication by a primitive third root of unity in a finite field. More generally, it gives the following construction.

\begin{lemma}\label{leCOConstruction}
Let $R$ be a finite unital ring of order $q\equiv1\pmod6$. If there exists an element $\omega\in R$ such that
\[
1+\omega+\omega^2=0,
\]
then there exists a CO$(q)$.
\end{lemma}
\begin{proof}
Let $G$ be the additive group of $R$, and let $f$ be left multiplication by $\omega$. Since $q$ is relatively prime to $3$, the element $3$ is a unit in $R$.

The identities $f(-a)=-f(a)$ and $a+f(a)+f^2(a)=0$ are immediate. Moreover,
\[
(\omega-1)\frac{\omega^2-1}{3}=1,
\]
so $\omega-1$ is a unit. Consequently, the map $a\mapsto f(a)-a$ is bijective. The identity $\omega^3=1$, together with the invertibility of $\omega-1$, shows that $f$ has order $3$ and that $0$ is its only fixed point.

Thus the nonzero elements of $G$ are partitioned into $f$-orbits of length $3$. Since $f(-a)=-f(a)$, these orbits occur in pairs $\{O,-O\}$. Choose one orbit from each pair and let $G^+$ be the union of the chosen orbits. Then
\[
G=\{0\}\mathbin{\dot\cup}G^+\mathbin{\dot\cup}(-G^+)
\qquad\text{and}\qquad
f(G^+)=G^+.
\]

It remains to verify properties~(6) and~(7). By the additivity of $f$,
\[
f^2(a)+f(b-a)=f(f(a)-a)+f(b),
\]
which is bijective as a function of $a$, since both $f$ and $a\mapsto f(a)-a$ are bijective.

Similarly, the two sets in property~(7) reduce to
\[
\{a-f(a):a\in G^+\}
\quad\text{and}\quad
\{f^2(a)-f(a):a\in G^+\}.
\]
Each set has cardinality $(q-1)/2$, because the maps $a\mapsto a-f(a)$ and $a\mapsto f^2(a)-f(a)$ are bijective on $G$. Suppose that the two sets intersect. Then, for some $a_1,a_2\in G^+$,
\[
a_1-f(a_1)=f^2(a_2)-f(a_2)
=f(-a_2)-f^2(-a_2).
\]
The right-hand side is $c-f(c)$ with $c=f(-a_2)$. Since $a\mapsto a-f(a)$ is bijective, we obtain $a_1=f(-a_2)$. This is impossible, because $a_1\in G^+$ while $f(-a_2)\in-G^+$. Hence the two sets are disjoint and therefore partition $G\setminus\{0\}$.
\end{proof}

\begin{cor}\label{cor:COexistence}
Let $q=\prod_i q_i$, where every $q_i$ is a prime power congruent to $1$ modulo $6$. Then there exists a CO$(q)$.
\end{cor}
\begin{proof}
Apply Lemma~\ref{leCOConstruction} to the ring
\[
R=\prod_i\FF{q_i}.
\]
Every field $\FF{q_i}$ contains a primitive third root of unity. Taking $\omega$ componentwise gives an element of $R$ satisfying $1+\omega+\omega^2=0$.
\end{proof}

We record two consequences of Definition~\ref{defnCO} that will be used repeatedly.

\begin{lemma}\label{leCOequ1}
Suppose that $(G,f,G^+)$ is a CO$(q)$. Then the following statements hold.
\begin{enumerate}
\item If $u,v,w\in G$ satisfy $u+v+w=0$, then the equations
\[
a+b=u,\qquad f(a)+f^2(b)=v,\qquad f^2(a)+f(b)=w
\]
have a unique solution $(a,b)\in G^2$.
\item For every pair of distinct elements $u,v\in G$, there exists a unique pair $(a,b)\in G^+\times G$ such that
\[
\{u,v\}=\{a+b,f^2(a)+b\}.
\]
\end{enumerate}
\end{lemma}
\begin{proof}
For (1), put $b=u-a$. By Definition~\ref{defnCO}(6), there is a unique $a\in G$ such that
\[
f^2(a)+f(u-a)=w.
\]
With $b=u-a$, the first and third equations hold. Definition~\ref{defnCO}(5) gives
\[
f(a)+f^2(b)=-(a+b)-(f^2(a)+f(b))=-u-w=v,
\]
so the second equation also holds. Uniqueness follows from the bijectivity in property~(6).

For (2), consider
\[
h:G\longrightarrow G,\qquad h(a)=a-f^2(a).
\]
Since
\[
h(a)=f(f^2(a))-f^2(a),
\]
the map $h$ is bijective by Definition~\ref{defnCO}(1). It is also odd. Therefore
\[
\{\pm h(a):a\in G^+\}=G\setminus\{0\},
\]
and each nonzero element occurs exactly once in this form.

Let $u\ne v$. There is a unique $a\in G^+$ such that $u-v=\pm h(a)$. If $u-v=h(a)$, take $b=u-a$; if $u-v=-h(a)$, take $b=u-f^2(a)$. In either case,
\[
\{u,v\}=\{a+b,f^2(a)+b\}.
\]
The uniqueness of $a$ and then of $b$ proves the assertion.
\end{proof}

\section{The construction}

Suppose that $(G,f,G^+)$ is a CO$(q)$. Let
\[
Y=G\cup\{\infty_1,\infty_2\},
\]
and let $\{\mathcal D_c:c\in G\}$ be an LKTS$(q+2)$ on $Y$. For every $c\in G$, fix a resolution
\[
\mathcal D_c=\mathbin{\dot\bigcup}_{k=1}^{(q+1)/2}Q_{c,k}.
\]
We construct an LKTS$(7q+2)$ on the point set
\[
X=(G\times\FF{7})\cup\{\infty_1,\infty_2\}.
\]

For $(c,x)\in G\times\FF{7}$, define a permutation $\tau_{c,x}$ of $X$ by fixing $\infty_1$ and $\infty_2$ and setting
\[
\tau_{c,x}(a,y)=(a+c,y+x)
\qquad(a\in G,\ y\in\FF{7}).
\]
We use the same notation for the induced action on blocks and collections of blocks.

For each $a\in G\setminus\{0\}$, let $P_a$ be the set of the following $14$ triples in $G\times\FF{7}$:
\[
\begin{array}{ll}
\{(-a,0),(f(a),1),(f^2(a),6)\} & \{(a,1),(f^2(a),1),(-a,5)\} \\
\{(f(a),0),(-f^2(a),1),(-a,6)\} & \{(a,3),(f^2(a),3),(-a,1)\} \\
\{(-f(a),0),(f^2(a),2),(a,5)\} & \{(f(a),2),(a,2),(-f(a),3)\}\\
\{(f^2(a),0),(-a,2),(-f(a),5)\} & \{(f(a),6),(a,6),(-f(a),2)\}\\
\{(-f^2(a),0),(a,4),(f(a),3)\} & \{(f^2(a),4),(f(a),4),(-f^2(a),6)\}\\
\{(a,0),(-f(a),4),(-f^2(a),3)\} & \{(f^2(a),5),(f(a),5),(-f^2(a),4)\}\\
\\
\{(-f(a),1),(-f^2(a),2),(-a,4)\} & \{(-f(a),6),(-f^2(a),5),(-a,3)\}.
\end{array}
\]
A direct check shows that $P_a$ is a partition of the set $\{\pm a,\pm f(a),\pm f^2(a)\}\times\FF{7}$. We choose a system of representatives $\{a_i\}_{1\leq i\leq (q-1)/6}$ from the orbits of $f$ acting on $G^+$, and use them to construct a parallel class. The following lemma is immediate.

\begin{lemma}\label{lem:baseparallel}
For $j\in\FF{3}$, let
\[
\begin{split}
P_{j,0}={}&\bigcup_{i=1}^{(q-1)/6}P_{f^j(a_i)}\\
&\cup\bigl\{\{\infty_1,(0,2^{j+1}),(0,-2^{j+2})\},
\{\infty_2,(0,2^{j+2}),(0,-2^{j+1})\},\\
&\hspace{42mm}\{(0,0),(0,2^j),(0,-2^j)\}\bigr\},
\end{split}
\]
where the second coordinates are calculated in $\FF{7}$. Then $P_{j,0}$ is a parallel class of $X$.
\end{lemma}
\begin{proof}
The sets $P_{f^j(a_i)}$ partition $(G\setminus\{0\})\times\FF{7}$, because the sets
\[
\{\pm a_i,\pm f(a_i),\pm f^2(a_i)\}
\]
partition $G\setminus\{0\}$. The three remaining blocks partition
\[
\{\infty_1,\infty_2\}\cup(\{0\}\times\FF{7}).
\]
Hence $P_{j,0}$ is a partition of $X$.
\end{proof}

We next develop each of these three parallel classes into $q$ parallel classes. For $b\in G$, define a permutation $g_b$ of $X$ that fixes $\infty_1$ and $\infty_2$ and acts on $G\times\FF{7}$ as follows:
\begin{align*}
g_b(a,0)&=(a,0),\\
g_b(a,1)&=(a+b,1) & g_b(a,2)&=(a+f^2(b),2) & g_b(a,4)&=(a+f(b),4), \\
g_b(a,6)&=(a-b,6) & g_b(a,5)&=(a-f^2(b),5) & g_b(a,3)&=(a-f(b),3).
\end{align*}

For later use, write $\sigma_x(b)$ for the amount by which $g_b$ shifts
the first coordinate in the layer $G\times\{x\}$. Thus
\[
\begin{aligned}
\sigma_1(b)&=b, & \sigma_2(b)&=f^2(b), & \sigma_4(b)&=f(b),\\
\sigma_{-x}(b)&=-\sigma_x(b) &&& &(x\in\{1,2,4\}).
\end{aligned}
\]
In particular, each map $b\mapsto\sigma_x(b)$ is a bijection of $G$.
Using Definition~\ref{defnCO}(2) and (3), a direct check also gives
\begin{equation}\label{eq:sigmashift}
\sigma_{-2x}(b)=-f^2(\sigma_x(b))
\qquad(x\in\FF{7}^*).
\end{equation}

For $j\in\FF{3}$ and $b\in G$, let
\[
P_{j,b}=g_b(P_{j,0}).
\]
This gives $3q$ parallel classes. We now construct another $(q+1)/2$ partial
parallel classes.

\begin{lemma}\label{lem:partialparallel}
For every $a\in G$, the following set of $2q$ blocks
\[
P^*_a=\left\{\begin{array}{c}g_b(\{(a,1),(f(a),2),(f^2(a),4)\})\\
g_b(\{(a,6),(f(a),5),(f^2(a),3)\})\end{array}
\middle| b\in G\right\}
\]
is a partition of $G\times\FF{7}^*$.
\end{lemma}
\begin{proof}
As $b$ ranges over $G$, the first family of blocks partitions
\[
G\times\{1,2,4\}.
\]
Indeed, in the three layers the first coordinates are respectively
\[
a+b,\qquad f(a)+f^2(b),\qquad f^2(a)+f(b),
\]
and each of the three expressions runs through $G$ exactly once. The same argument shows that the second family partitions $G\times\{3,5,6\}$. The two families are disjoint, and the result follows.
\end{proof}

We next characterize the blocks, and hence the pairs, contained in the classes $P_{j,b}$ and the partial classes $P^*_a$.

\begin{lemma}\label{lePartialBlocks}
The $3q$ parallel classes $P_{j,b}$, where $j\in\FF{3}$ and $b\in G$, together with the $(q+1)/2$ partial parallel classes $P^*_a$, where $a\in G^+\cup\{0\}$, contain precisely the following blocks:
\begin{enumerate}
\item[Type I.] The $3q^2$ blocks
\[
\{(u,0),(v,x),(u-v,-x)\},
\qquad u,v\in G,\quad x\in\{1,2,4\}.
\]
\item[Type II.] The $2q^2$ blocks
\[
\{(u+v,1),(f(u)+f^2(v),2),(f^2(u)+f(v),4)\}
\]
and
\[
\{(u-v,6),(f(u)-f^2(v),5),(f^2(u)-f(v),3)\},
\qquad u,v\in G.
\]
\item[Type III.] The $3q(q-1)$ blocks
\[
\{(u+v,x),(f^2(u)+v,x),(-u-f^2(v),-2x)\},
\]
where $u\in G^+$, $v\in G$, and $x\in\FF{7}^*$.
\item[Type IV.] The $6q$ blocks
\[
\{\infty_i,(v,x),(-f^2(v),-2x)\},
\qquad v\in G,\quad x\in\FF{7}^*,
\]
where $i=1$ for $x\in\{1,2,4\}$ and $i=2$ for $x\in\{3,5,6\}$.
\end{enumerate}
Consequently, these $8q^2+3q$ blocks contain every pair of points of $X$ exactly once, except for the pairs whose two points both lie in
\[
H_0=\{\infty_1,\infty_2\}\cup(G\times\{0\}).
\]
\end{lemma}
\begin{proof}
The classes and partial classes under consideration contain
\[
3q\cdot\frac{7q+2}{3}+\frac{q+1}{2}\cdot2q=8q^2+3q
\]
blocks, which is also the total number of blocks in Types I--IV. It therefore suffices to show that every listed block occurs.

For Type I, begin with a block
\[
\{(u,0),(-f(u),x),(-f^2(u),-x)\},
\qquad x\in\{1,2,4\}.
\]
For $u=0$, this is one of the three additional blocks in some $P_{j,0}$.
For $u\ne0$, the element $u$ belongs to one of the sets
\[
\{\pm a_i,\pm f(a_i),\pm f^2(a_i)\},
\]
and inspection of the definition of $P_a$ shows that the block belongs to
one of the classes $P_{j,0}$. Given $v\in G$, choose the unique $b\in G$
such that $\sigma_x(b)=v+f(u)$. Its image under $g_b$ is
\[
\{(u,0),(v,x),(u-v,-x)\},
\]
where the last coordinate follows from $u+f(u)+f^2(u)=0$.

The Type II blocks are the images under $g_v$ of
\[
\{(u,1),(f(u),2),(f^2(u),4)\}
\quad\text{and}\quad
\{(u,6),(f(u),5),(f^2(u),3)\}.
\]
If $u\in G^+\cup\{0\}$, these are blocks of $P_u^*$. Otherwise, write
$u=-f^j(a_i)$, where the exponent is read modulo $3$. Both base blocks then
occur in $P_{f^{j-1}(a_i)}$, and hence in one of the classes $P_{j-1,0}$.

For Type III, the definition of the sets $P_a$ shows that, for every
$u\in G^+$ and $x\in\FF{7}^*$, a class $P_{j,0}$ contains the block
\[
\{(u,x),(f^2(u),x),(-u,-2x)\}.
\]
Given $v\in G$, choose the unique $b$ such that $\sigma_x(b)=v$. By
\eqref{eq:sigmashift}, the image of this block under $g_b$ is
\[
\{(u+v,x),(f^2(u)+v,x),(-u-f^2(v),-2x)\}.
\]
Finally, the six base blocks of Type IV are the blocks containing an infinity
in the definitions of $P_{j,0}$. Applying the same argument and
\eqref{eq:sigmashift} gives all Type IV blocks. This proves the first
assertion.

It remains to check the pairs covered by these blocks. Type I contains every pair of either of the forms
\[
\{(u,0),(v,x)\},
\qquad
\{(u,x),(v,-x)\},
\]
where $u,v\in G$ and $x\in\FF{7}^*$. Type IV contains every pair
\[
\{\infty_i,(v,x)\},
\qquad i\in\{1,2\},\quad v\in G,\quad x\in\FF{7}^*.
\]

We next consider pairs between the layers $x$ and $-2x$. The Type III and Type IV blocks contain the following pairs:
\[
\{(v,x),(-f^2(v),-2x)\},
\]
\[
\{(u+v,x),(-u-f^2(v),-2x)\},
\]
and
\[
\{(f^2(u)+v,x),(-u-f^2(v),-2x)\},
\]
where $u\in G^+$ and $v\in G$. In the last two families, replace $v$ by $-f(u+v)$. The second coordinate then becomes $v$, while the first coordinates become, respectively,
\[
u-f(u+v)
\quad\text{and}\quad
f^2(u)-f(u+v).
\]
For fixed $v$, the first family contributes the point $-f(v)$ in the layer $x$. Relative to this point, the other two families contribute the differences
\[
u+f(v)-f(u+v)
\quad\text{and}\quad
f^2(u)+f(v)-f(u+v).
\]
By Definition~\ref{defnCO}(7), these differences partition $G\setminus\{0\}$. Hence every pair between the layers $x$ and $-2x$ occurs.

By Lemma~\ref{leCOequ1}(2), Type III contains every pair of distinct points in a common nonzero layer. By Lemma~\ref{leCOequ1}(1), Type II contains every pair between two distinct layers both belonging to $\{1,2,4\}$, and similarly every pair between two distinct layers both belonging to $\{3,5,6\}$.

These cases include every pair not wholly contained in $H_0$. Their number is
\[
\binom{7q+2}{2}-\binom{q+2}{2}=24q^2+9q.
\]
The listed blocks contain
\[
3(8q^2+3q)=24q^2+9q
\]
pairs, counted with multiplicity. Since every admissible pair occurs at least once, every such pair occurs exactly once.
\end{proof}

Let
\[
H_x=\{\infty_1,\infty_2\}\cup(G\times\{x\})
\qquad(x\in\FF{7}),
\]
and let
\[
\ell:Y\longrightarrow H_0
\]
be the bijection that fixes $\infty_1$ and $\infty_2$ and sends $a\in G$ to
$(a,0)$. Enumerate the elements of $G^+\cup\{0\}$ as
\[
d_1,d_2,\ldots,d_{(q+1)/2}.
\]

\begin{theorem}\label{thm:oneKTS}
For every $c\in G$, the classes
\[
P^*_{c,k}=\ell(Q_{c,k})\cup\tau_{c,0}(P^*_{d_k}),
\qquad 1\leq k\leq\frac{q+1}{2},
\]
together with the classes
\[
\tau_{c,0}(P_{j,b}),
\qquad j\in\FF{3},\quad b\in G,
\]
form a resolution of a KTS$(7q+2)$ on $X$.
We denote this KTS by $\mathcal B_c$.
\end{theorem}
\begin{proof}
For each $k$, the class $\ell(Q_{c,k})$ partitions $H_0$, while
$\tau_{c,0}(P^*_{d_k})$ partitions $G\times\FF{7}^*$ by
Lemma~\ref{lem:partialparallel}. Hence $P^*_{c,k}$ is a parallel class of
$X$. The classes $\tau_{c,0}(P_{j,b})$ are parallel classes by
Lemma~\ref{lem:baseparallel} and the definition of $P_{j,b}$.

The embedded KTS $\ell(\mathcal D_c)$ covers every pair in $H_0$ exactly
once. By Lemma~\ref{lePartialBlocks}, the collection $\mathcal E$ covers
every remaining pair exactly once. Since $\tau_{c,0}$ preserves $H_0$, the
same is true of $\tau_{c,0}(\mathcal E)$. Consequently, every pair of points
of $X$ occurs in exactly one of the listed blocks.

Finally, the number of parallel classes is
\[
3q+\frac{q+1}{2}=\frac{7q+1}{2},
\]
as required.
\end{proof}

We next show that the blocks used outside $H_0$ form a transversal for the
translation action. Put
\[
\mathcal E=
\left(\bigcup_{j\in\FF{3}}\ \bigcup_{b\in G}P_{j,b}\right)
\cup
\left(\bigcup_{a\in G^+\cup\{0\}}P^*_a\right).
\]
Thus $\mathcal E$ consists precisely of the blocks of Types I--IV in
Lemma~\ref{lePartialBlocks}.

\begin{lemma}\label{lem:orbittransversal}
Every triple of $X$ that is not contained in any $H_x$ can be written uniquely
in the form
\[
\tau_{c,x}(B),
\qquad (c,x)\in G\times\FF{7},\quad B\in\mathcal E.
\]
\end{lemma}
\begin{proof}
We distinguish the possible multiplicities of the second coordinates of the
finite points.

Suppose first that the triple consists of three finite points with distinct
second coordinates. The translation orbits of the three-subsets of $\FF{7}$
have the following five representatives:
\[
\{0,1,6\},\quad \{0,2,5\},\quad \{0,3,4\},\quad
\{1,2,4\},\quad \{3,5,6\}.
\]
The first three are the second-coordinate sets of the Type I blocks, and the
last two are those of the Type II blocks.

After the unique translation that puts the second coordinates into
$\{0,x,-x\}$, with $x\in\{1,2,4\}$, write the first coordinates as $r,s,t$
in the layers $0,x,-x$, respectively. There is a unique $c\in G$ for which
\[
t+c=(r+c)-(s+c),
\]
namely $c=r-s-t$. The translated triple is therefore a unique Type I block.

Now suppose that the second coordinates have been translated to
$\{1,2,4\}$. If the corresponding first coordinates are $r,s,t$, there is a
unique $c\in G$ such that
\[
(r+c)+(s+c)+(t+c)=0,
\]
because multiplication by $3$ is a bijection of $G$. Lemma~\ref{leCOequ1}(1)
then gives unique $u,v\in G$ such that the translated triple is
\[
\{(u+v,1),(f(u)+f^2(v),2),(f^2(u)+f(v),4)\}.
\]
For the second-coordinate set $\{3,5,6\}$, order the layers as
$6,5,3$ and apply the same argument with the second parameter replaced by
its negative. Thus the assertion holds when the three second coordinates
are distinct.

Suppose next that the triple consists of three finite points, exactly two of
which have the same second coordinate. There is a unique translation and a
unique $x\in\FF{7}^*$ that put its second-coordinate multiset into
\[
\{x,x,-2x\}.
\]
Write the first coordinates of the two points in the layer $x$ as $r,s$, and
that of the remaining point as $t$. By Lemma~\ref{leCOequ1}(2), there is a
unique pair $(u,b)\in G^+\times G$ such that
\[
\{r,s\}=\{u+b,f^2(u)+b\}.
\]
After translation by $c\in G$, the parameter $b$ is replaced by $b+c$. The
remaining point has the required first coordinate precisely when
\[
t+c=-u-f^2(b+c).
\]
The left-hand side can be solved uniquely for $c$, since, on putting
$z=b+c$, we have
\[
c+f^2(b+c)=z-b+f^2(z)=-f(z)-b,
\]
and $f$ is a permutation. This gives a unique Type III block.

Finally, suppose that the triple contains one point $\infty_i$ and two finite
points. Since the triple is not contained in any $H_x$, the two finite points
have distinct second coordinates. Put
\[
R_1=\{1,2,4\},\qquad R_2=\{3,5,6\}.
\]
There is a unique ordering of the two finite points and a unique translation
of their second coordinates that puts them into the layers $x,-2x$, with
$x\in R_i$. Write their first coordinates in these layers as $r,s$. We seek
$c,v\in G$ satisfying
\[
r+c=v,\qquad s+c=-f^2(v).
\]
Eliminating $c$ gives
\[
v+f^2(v)=r-s,
\]
or equivalently $-f(v)=r-s$. This has a unique solution $v$, and then $c$ is
also uniquely determined. The resulting block is of Type IV.

These cases exhaust all triples not contained in one of the sets $H_x$.
The uniqueness assertions above also give uniqueness of the representative
block. For completeness, the translating element is unique as well: the
stabilizer of a triple of finite points has order dividing $6$, while the
stabilizer of a triple containing one infinity has order dividing $2$; both
stabilizers are subgroups of the group $G\times\FF{7}$, whose order is
relatively prime to $6$. Hence both stabilizers are trivial.
\end{proof}

For $c\in G$ and $x\in\FF{7}$, define
\[
\mathcal B_{c,x}=\tau_{0,x}(\mathcal B_c).
\]

\begin{theorem}\label{thm:main}
If there exist an LKTS$(q+2)$ and a CO$(q)$, then there exists an
LKTS$(7q+2)$.
\end{theorem}
\begin{proof}
By Theorem~\ref{thm:oneKTS}, every $\mathcal B_{c,x}$ is a KTS$(7q+2)$.
Its blocks inside $H_x$ are $\tau_{0,x}(\ell(\mathcal D_c))$, while its
remaining blocks form $\tau_{c,x}(\mathcal E)$. We show that
\[
\{\mathcal B_{c,x}:c\in G,\ x\in\FF{7}\}
\]
partitions the triples of $X$.

Let $T$ be a triple contained in some $H_x$. The value of $x$ is uniquely
determined by $T$, since two distinct sets $H_x$ intersect in only the two
points at infinity. The triple $\tau_{0,-x}(T)$ is contained in $H_0$, and its
inverse image under $\ell$ is a triple of $Y$. Since
$\{\mathcal D_c:c\in G\}$ is an LKTS$(q+2)$, this triple belongs to a unique
$\mathcal D_c$. Hence $T$ belongs to the internal part of a unique
$\mathcal B_{c,x}$.

If $T$ is not contained in any $H_x$, Lemma~\ref{lem:orbittransversal} gives
a unique block $B\in\mathcal E$ and a unique pair $(c,x)\in G\times\FF{7}$
such that $T=\tau_{c,x}(B)$. The external part of $\mathcal B_{c,x}$ is
exactly $\tau_{c,x}(\mathcal E)$, so again $T$ occurs in a unique member of
the family.

Thus the $7q=(7q+2)-2$ KTSs $\mathcal B_{c,x}$ are mutually disjoint and
partition all triples of $X$. They therefore form an LKTS$(7q+2)$.
\end{proof}

\begin{cor}\label{cor:recursive}
Let $q=\prod_i q_i$, where every $q_i$ is a prime power congruent to $1$
modulo $6$. If there exists an LKTS$(q+2)$, then there exists an
LKTS$(7q+2)$.
\end{cor}
\begin{proof}
Combine Corollary~\ref{cor:COexistence} with Theorem~\ref{thm:main}.
\end{proof}

\begin{cor}\label{cor:neworders}
There exist LKTSs of orders
\[
93,\quad177,\quad219,\quad261,\quad303.
\]
\end{cor}
\begin{proof}
Apply Corollary~\ref{cor:recursive} with
\[
q=13,\quad25,\quad31,\quad37,\quad43,
\]
respectively. The required input orders $q+2$ are $15,27,33,39$, and $45$.
The LKTSs of orders $15,33$, and $39$ are among the direct constructions
listed in Table~\ref{tab:direct}; those of orders $27$ and $45$ follow from
Denniston's tripling construction~\cite{MR535159}, starting from orders $9$
and $15$, respectively.
\end{proof}

\section{Concluding remarks}

We conclude by recording the effect of the known constructions on the small
orders. Combining Table~\ref{tab:direct}, Denniston's tripling construction,
the construction of Wang and Shi~\cite{MR3612438}, the family obtained by Xu
and Ji~\cite{XuJi2021}, and Corollary~\ref{cor:neworders}, an LKTS$(v)$ is
known to exist for every $v<200$ satisfying $v\equiv3\pmod6$, with the two
possible exceptions being $57$ and $87$.

These two orders are now the smallest remaining cases of the existence
problem. They are also resistant to the most immediate finite-field methods:
writing them as $q+2$ gives $q=55$ and $q=85$, neither of which is a prime
power or a product of prime powers congruent to $1$ modulo $6$.

A natural short-term objective is therefore to construct LKTS$(57)$ and
LKTS$(87)$ directly, or to find recursive constructions designed specifically
to reach them. For example, a construction of the form
\[
\operatorname{LKTS}(v)\longrightarrow\operatorname{LKTS}(4v-3)
\]
would give order $57$ from order $15$, while a construction of the form
\[
\operatorname{LKTS}(v)\longrightarrow\operatorname{LKTS}(4v+3)
\]
would give order $87$ from order $21$. Such recursions may require an
intermediate object weaker than an LKTS, in the same way that overlarge sets,
Kirkman frames, systems with holes, and Steiner quadruple systems with
resolvable derived designs have been used in earlier constructions. Direct
computer search, preferably with a prescribed automorphism group, also seems
realistic for these two isolated orders.

For the general existence problem, the existing literature suggests two
complementary routes. One is to continue the auxiliary-design program: obtain
sufficiently flexible large sets with holes, overlarge sets, and resolvable
derived designs, and combine them through $3$-wise balanced designs and
fundamental constructions. Several conditional reductions already show that a
large part of the problem can be concentrated into a finite collection of
small auxiliary designs; see, for example, Lei~\cite{Lei2004} and the survey
of Chang and Zhou~\cite{MR3007142}. The other possible route is an asymptotic
decomposition theorem. An LKTS$(v)$ may be regarded as a decomposition of the
complete $3$-uniform hypergraph into Steiner triple systems, together with a
second decomposition of every constituent system into perfect matchings.
Iterative absorption has proved powerful for many fixed hypergraph
decomposition problems~\cite{GlockKuhnLoOsthus2023}, and suggests a possible
starting point for such an asymptotic approach. However, the constituent
Steiner triple systems here grow with $v$, and preserving the nested
resolvability is a substantial additional constraint. An asymptotic theorem for all sufficiently large admissible orders,
combined with direct constructions for the remaining finite set, would give a
complete solution of the LKTS existence problem.

\bibliographystyle{siam}
\bibliography{designs,designs_additions_v2}

\end{document}